\documentclass[a4paper]{amsart}

\usepackage{amsaddr}
\usepackage{amssymb}

\usepackage{a4wide}
\usepackage{amsmath,amssymb,amsthm}

\def\a{\alpha} \def\b{\beta} \def\d{\delta}  
 \def\s{\sigma}  
\def\I{\mathbb{I}}

\def\<{\langle} \def\>{\rangle}
\def\inv{^{-1}}

\renewcommand\ge{\geqslant}
\renewcommand\le{\leqslant}

\newcommand\eg{e.g.}

\newcommand\forget[1]{}

\DeclareMathOperator{\spann}{span}

\DeclareMathOperator{\lt}{L}

\DeclareMathOperator{\End}{End}

\renewcommand{\Im}{\mathop{\rm Im}\nolimits}

\renewcommand{\ker}{\mathop{\rm ker}\nolimits}
\newcommand{\chr}{\mathop{\rm char}\nolimits}

\newenvironment{smatrix}{\left(\begin{smallmatrix}}{\end{smallmatrix}\right)}

\newcommand\pmatr[1]{\begin{pmatrix}#1\end{pmatrix}}

\DeclareMathOperator{\tr}{tr}

\newtheorem{thm}{Theorem}[section]
\newtheorem{lma}[thm]{Lemma}
\newtheorem{prop}[thm]{Proposition}
\newtheorem{cor}[thm]{Corollary}

\theoremstyle{definition}
\newtheorem{dfn}[thm]{Definition}
\theoremstyle{remark}

\theoremstyle{definition}
\newtheorem{ex}[thm]{Example}

\renewcommand{\theenumi}{\alph{enumi}}

\newcommand{\V}{\mathcal{CD}}

\newcommand{\n}{\mathop{\rm n}\nolimits}
\newcommand{\ep}{\hspace*{-.125em}\times\hspace*{-.125em}}

\title[Von-Neumann finiteness and reversibility in non-associative algebras]{Von-Neumann finiteness and reversibility in some classes of non-associative algebras}

\begin{document}

\author{Erik Darp\"{o}}
\address[Darp\"o]{Graduate School of Mathematics, 
Nagoya University, 
Furo-cho, Chikusa-ku, 
Nagoya, 464-8602, Japan}
\email{darpo@math.nagoya-u.ac.jp}

\author{Patrik Nystedt}
\address[Nystedt]{University West,
Department of Engineering Science, 
SE-46186 Trollh\"{a}ttan, 
Sweden}
\email{patrik.nystedt@hv.se}

\subjclass[2010]{17D05, 16K20, 17C55, 16W10, 16U99}

\keywords{von-Neumann finite, Dedekind finite, directly finite, reversible, quadratic
  algebra, involutive algebra}

\begin{abstract}
We investigate criteria for von-Neumann finiteness and reversibility in some classes of
non-associative algebras. 
Types of algebras that are studied include alternative, flexible, quadratic and involutive
algebras, as well as algebras obtained by the Cayley--Dickson doubling process. 
Our results include precise criteria for von-Neumann finiteness and reversibility of
involutive algebras in terms of isomorphism types of their $3$-dimensional subalgebras.
\end{abstract}

\maketitle

\section{Introduction}

A unital ring $A$ is called \emph{von-Neumann finite}
(or Dedekind finite, or weakly $1$-finite,
or affine finite, or directly finite, or inverse symmetric)
if every one-sided inverse in $A$ also is two-sided,
in other words, if for all $a,b \in A$ satis\-fying $ab = 1$,
the relation $ba = 1$ also holds.
Many different classes of {\it associative} rings have been shown 
to be von-Neumann finite, for instance
noetherian, 
self-injective and
PI-rings \cite{lam2007}.
Also group rings over fields of characteristic zero
\cite{kaplansky1969,montgomery1969,passman1971},
or, more generally,
endomorphism rings of permutation modules over such rings \cite{lundstrom2003,lundstrom2004}, 
have been shown to be von-Neumann finite. 
Von-Neumann finiteness for group rings of positive characteristic has,
however, remained an open problem (for partial results, 
see \cite{aop02} and \cite{elek2004}).
Of course, not all associative unital rings are von-Neumann finite;
for instance if $V$ is a vector space, then it is easy to see that $\End(V)$ 
is von-Neumann finite if and only if $V$ is finite dimensional. 

Following Cohn \cite{cohn99}, we say that $A$ is 
{\it reversible} if for all $a,b \in A$ satisfying $ab = 0$,
the relation $ba = 0$ also holds.
It is easy to see that the class of associative reversible rings
is properly contained in the class of associative von-Neumann finite rings.
Indeed, suppose that $A$ is an associative reversible ring
and $ab=1$ for some $a,b \in A$. Then, since $(ba-1)b = 0$,
we get that $b(ba-1)=0$ that is $b^2 a = b$ which implies that
$ba = abba = ab = 1$. 
Moreover, if $V$ is a vector space
satisfying $1 < {\rm dim}(V) < \infty$,
then ${\rm End}(V)$ is von-Neumann finite but not reversible.

In this article, we consider von-Neumann finiteness and reversibility for some classes of
unital rings which are not associative -- apparently a new line of investigation.
%
Our motivation for doing this is two-fold. 
First, we wish to see to which extent patterns and phenomena from the associative
context reproduce in more general classes of algebras, and how far the
associative theory can be generalised.
Second, the notions of von-Neumann finiteness and reversibility are relevant by
themselves in the structure theory of non-associative algebras. 
Particularly, a long-standing problem has been to understand the structure of
\emph{quadratic division algebras}
(see, \eg, \cite{cdkr99,nform,revisited,dieterichohman,lindberg04,osborn62}). 
The division property being notoriously difficult to study directly, investigation of
related classes of quadratic algebras can provide insight leading to a greater
understanding of division algebras.
The reversible and von-Neumann finite quadratic algebras are two such classes. While both
properly contain the class of quadratic division algebras, they are also sufficiently
narrowly defined to make discernible certain particular structural properties of
these classes of algebras. 

Let $F$ be a field. 
An \emph{$F$-algebra} is a vector space $A$ over $F$ endowed with a bilinear multiplication map $A\times A\to A,\:(a,b)\mapsto ab$. An $F$-algebra $A$ is said to be \emph{unital} if it possesses a multiplicative identity element $1=1_A$.
In this article, we will consider the problem of characterising von-Neumann finiteness and reversibility for the classes of non-associative algebras defined below.

\begin{dfn} An $F$-algebra $A$ is said to be
  \begin{enumerate}
  \item \emph{alternative} if it satisfies the identities $a^2b=a(ab)$
    and $ab^2=(ab)b$ for all $a,b\in A$. 
    Equivalently, every subalgebra of $A$ generated by at most two elements is associative \cite[Theorem~3.1]{schafer1995}.
  \item \emph{flexible} if it satisfies the identity $a(ba)=(ab)a$ for all $a,b\in A$.
  \item \emph{quadratic} if it is unital and the elements $1,a,a^2$ are
    linearly dependent for all $a\in A$.
  \item \emph{involutive} if it is unital and there exists an anti-automorphism $\sigma$
    of $A$ such that $\s^2=\I_A$, $a+\s(a)\in F1$ and $a\s(a)\in F1$ for all $a\in A$.
    We will use the notation $\s(a)=\bar{a}$ for the involution in $A$. 
    For any $a\in A$, the scalars $\tr(a)=a+\bar{a}$ and $\n(a)=a\bar{a}$ are
    called the \emph{trace} and the \emph{norm} of $A$, respectively. 
  \end{enumerate}
\end{dfn}

From here on, all algebras will be assumed to be unital (but not necessarily associative).
Throughout, $F$ denotes a field, and $A$ an $F$-algebra. 

To every quadratic form $q:V\to F$ on a vector space $V$ over $F$ is associated a
symmetric bilinear form $\<x,y\>=\<x,y\>_q=q(x+y)-q(x)-q(y)$. The \emph{radical} of $q$ is
the subspace $V^\perp=\{x\in V\mid \<x,V\>_q=0\}$ of $V$.
The form $q$ is said to be \emph{non-degenerate} if either $V^\perp=0$ or
$\dim(V^\perp)=1$ and $q(V^\perp)\ne0$ (the latter case occurs only for $\chr F =2$). 
A non-zero element $v\in V$ is called \emph{isotropic} if $q(v)=0$ and
\emph{anisotropic} otherwise. The form $q$ is said to be \emph{isotropic}, respectively
\emph{anisotropic}, if $V$ contains, respectively does not contain, an isotropic
element. A subspace $U$ of $V$ such that $q(U)=0$ is said to be \emph{totally isotropic}.

A \emph{Hurwitz algebra} is an algebra $A$ possessing a non-degenerate quadratic form
$\n:A\to F$ satisfying $\n(ab)=\n(a)\n(b)$ for all $a,b\in A$. 
The form $\n$ is uniquely determined by the algebra structure of $A$, and every
non-zero algebra morphism between Hurwitz algebras is orthogonal. 
A Hurwitz algebra $A$ has zero-divisors if and only if its quadratic form $\n$ is
isotropic. If this is the case, $A$ is said to be \emph{split}. 
Over any ground field, there are precisely three isomorphism classes of split Hurwitz
algebra, one in each dimension $2$, $4$ and $8$, consecutively embedded into each other.   
The $4$-dimensional split Hurwitz $F$-algebra is the $2\times2$ matrix algebra over $F$. 
A Hurwitz algebra is commutative if and only if its dimension is at most $2$, and
associative if and only if its dimension is at most $4$.
All Hurwitz algebras are alternative. 
For further details, we refer to \cite[Chapter~VIII]{involutions}.

We denote by $H=F^{2\times2}$ the $4$-dimensional split Hurwitz algebra over $F$, and by
$U\subset H$ the $3$-dimensional subalgebra of upper triangular matrices. 

Given an involutive algebra $B$ and a non-zero scalar $\mu$, a new algebra $\V_{\mu}(B)$
is constructed as follows: as a vector space, $\V_{\mu}(B)=B\times B$, with multiplication
defined by 
$$(a,b)(c,d) = (ac +\mu\bar{d}b ,  da +  b\bar{c})$$
for all $a,b,c,d\in B$. 
The algebra $\V_{\mu}(B)$ is said to be obtained from $B$ by the 
\emph{Cayley--Dickson doubling process}. 
It is again involutive, with involution given by 
$\overline{(a,b)} = (\bar{a},-b)$, and flexible if and only if $B$ is flexible
\cite[Sec.~6]{mccrimmon85}.
Commonly, $B$ is identified with the subalgebra $B\times0$ of $\V_{\mu}(B)$, and setting
$\ell=(0,1)\in B\times B$ gives an orthogonal decomposition of vector spaces
$\V_{\mu}(B)=B\oplus B\ell$.
 
Starting from the algebra $F$ with trivial involution, one obtains a
class of flexible involutive algebras of dimensions $2^m$, $m\ge1$, by successive
application of the Cayley--Dickson construction. In this paper, we shall refer to the
algebras constructed in this way as \emph{Cayley--Dickson algebras}. 
When $\chr F\ne2$, the Cayley--Dickson algebras of dimension at most eight are precisely
the Hurwitz algebras over $F$ \cite{albert42d}. (Hurwitz algebras over fields of
  characteristic $2$ can also be constructed by the Cayley--Dickson process by starting
  from a so-called \emph{quadratic \'etale algebra}, but this construction will play no
  role in the present paper.)

The present article is organised as follows. In Section~\ref{summary}, we summarise our
main results. Section~\ref{preliminaries} contains some general background information and
lemmata, and Section~\ref{proofs} the proofs of the main results. Finally, in
Section~\ref{examples} we give some examples illuminating different aspects of the
theory. 

A note on conventions: we use words as \emph{non-commutative} and \emph{non-associative}
in the strict sense: a non-commutative algebra is one that does \emph{not} satisfy the
commutative law, etc.

\section{Summary of our results}
\label{summary}

Our first theorem, concerning alternative algebras, is rather straightforward from existing theory.

\begin{thm}[Propositions~\ref{alt1}--\ref{alt3}] \label{alt}
\rule{0pt}{0pt}
  \begin{enumerate}
  \item Every finite-dimensional alternative algbra is von-Neumann finite. 
    \label{alta}
  \item Every reversible alternative algebra is  von-Neumann finite. 
  \item A Hurwitz algebra $A$ is reversible if and only if either its quadratic form
    is anisotropic, or $\dim A\le2$. 
    \label{altd}
\end{enumerate}
\end{thm}

Note that, by Theorem~\ref{alt}\eqref{alta}, all Hurwitz algebras are von-Neumann finite.

The next theorem summarises properties of algebras that are both flexible and quadratic. Such algebras are in particular involutive, by Lemma~\ref{condition}\eqref{invcondition} in the next section, and hence equipped with a norm.

\begin{thm}[Propositions~\ref{0div}, \ref{anisotropic} and \ref{algclosedrev}]
\label{quadflex}
  \rule{0pt}{0pt}
  \begin{enumerate}
  \item Every algebra without zero divisors, that is either flexible or quadratic, is
    von-Neumann finite.
    \label{quadflex1}
  \item Assume that $\chr F\ne2$, and let $A$ be flexible and quadratic. If the norm of
    $A$ is non-degenerate on every $3$-dimensional subalgebra of $A$, then $A$ is
    von-Neumann-finite and reversible. 
    \label{quadflex2}
  \item
      Assume that $F$ is algebraically closed, $\chr F\ne2$, and that $A$ is flexible
    and quadratic. 
    Then $A$ is reversible if and only if either of the following conditions holds:
    \begin{enumerate} \renewcommand{\theenumii}{\roman{enumii}}
    \item $A$ is commutative;
    \item $A=F1\oplus V$, where $V$ is an anti-commutative ideal in $A$, and the
      linear map $\lt_u:V\to V,\:v\mapsto u v$ is nilpotent for all $u\in V$.
    \end{enumerate}
      \label{quadflex3}
  \end{enumerate}
\end{thm}

In relation to Theorem~\ref{quadflex}\eqref{quadflex1}, note that every algebra without
zero divisors is reversible.
We also point out that Theorem~\ref{quadflex}\eqref{quadflex2} in particular applies to
every flexible quadratic $F$-algebra with anisotropic norm. 
This immediately implies the first first part of our next theorem.

\begin{thm}[Corollary~\ref{doubledcor}, Corollary~\ref{nonrevcd}]
  \label{doubled}
  Assume that $\chr F\ne2$.
\begin{enumerate}
\item All Cayley--Dickson algebras with anisotropic norm are von-Neumann finite and
  reversible.%
  \label{doubled1}
\item A Cayley--Dickson algebra with isotropic norm is reversible if and only if its
  dimension is at most two.%
  \label{doubled2}
\end{enumerate}
\end{thm}

It is not clear to the authors if there exist Cayley--Dickson algebras with isotropic norm that are not von-Neumann finite. By Theorem~\ref{inv}\eqref{inv1} below, such an algebra would need to have a $3$-dimensional subalgebra that is neither commutative nor associative.
We also remark that when $\chr F\ne2$, Theorem~\ref{alt}\eqref{altd} is a consequence of
Theorem~\ref{doubled}\eqref{doubled2}.

\begin{thm}[Propositions~\ref{invvnf} and \ref{invrev}] \label{inv}
  Let $A$ be an involutive algebra, and assume that $\chr F\ne2$.
  \begin{enumerate}
  \item The algebra $A$ is von-Neumann finite if and only if every $3$-dimensional
    subalgebra of $A$ is either commutative or associative. 
    \label{inv1}
  \item The algebra $A$ is reversible if and only if every $3$-dimensional
    subalgebra of $A$ is commutative.
    \label{inv2}
  \end{enumerate}
\end{thm}

Actually, every non-commutative associative involutive algebra of dimension three over $F$ is isomorphic to the algebra $U$.

\section{Preliminaries}
\label{preliminaries}


First, observe that the classes of von-Neumann finite respectively reversible algebras are
closed under taking subalgebras.
Since matrix algebras are von-Neumann finite, and every finite-di\-men\-sion\-al
associative algebra can be embedded in a matrix algebra, we have the following (well
known) result.

\begin{lma}\label{fdass}
Every finite-dimensional associative $F$-algebra is von-Neumann finite. 
\end{lma}


\begin{lma} \label{vnfsubalg} 
 The following statements are equivalent.
  \begin{enumerate} \renewcommand{\theenumi}{\roman{enumi}}
  \item The algebra $A$ is von-Neumann finite (respectively, reversible);%
    \label{vnfsubalg1} 
  \item every subalgebra of $A$ generated by at most two elements is von-Neumann finite
    (reversible);%
    \label{vnfsubalg2} 
  \item any two elements $a,b\in A$ satisfying $ab\in F\setminus\{0\}$ ($ab=0$) commute
    with each other.
    \label{vnfsubalg3} 
  \end{enumerate}
\end{lma}

\begin{proof}
If $A$ is von-Neumann finite then so is every subalgebra of $A$. Hence
\eqref{vnfsubalg1}$\Rightarrow$\eqref{vnfsubalg2}.
Assuming \eqref{vnfsubalg2}, and taking $a,b\in A$ such that $ab=\mu\in F\setminus\{0\}$,
we have $(\mu\inv a)b=1$. By assumption, the subalgebra generated by $\mu\inv a$ and $b$
is von-Neumann finite, and hence $b(\mu\inv a)=1 = (\mu\inv a)b$, implying that
$ab=ba$. This shows that \eqref{vnfsubalg2}$\Rightarrow$\eqref{vnfsubalg3}.
Finally, if \eqref{vnfsubalg3} holds and $ab=1$, then $ba=ab=1$, so $A$ is von-Neumann
finite. 

The reversible case is analogous.
\end{proof}

We now turn our attention to involutive and quadratic algebras. 

\begin{lma} \label{quadvnf}
  Let $A$ be a quadratic algebra.
  \begin{enumerate}
  \item Let $a,b\in A$ and $B=\spann\{1,a,b\}$. If  $ab\in B$ then $B$ is a subalgebra of
    $A$. 
    \label{quadvnf1}
  \item The algebra $A$ is von-Neumann finite (respectively, reversible) if and
    only if every $3$-dimensional subalgebra of $A$ is von-Neumann finite (reversible).
    \label{quadvnf2}
  \item The algebra $A$ is von-Neumann finite and reversible if and only if every
    $3$-dimensional subalgebra of $A$ is commutative.
    \label{quadvnf3}
  \end{enumerate}
\end{lma}

\begin{proof}
\eqref{quadvnf1}
Since $A$ is quadratic, $a^2\in\spann\{1,a\}\subset B$, and similarly $b^2,(a+b)^2\in B$. 
Moreover, $ba = (a+b)^2-a^2-b^2-ab\in B$, so $B\subset A$ is a subalgebra.

\eqref{quadvnf2}
The ``only if'' part is immediate.
For the converse, assume that every $3$-dimensional subalgebra of $A$ is von-Neumann
finite, and let $a,b\in A$ such that $ab=1$. By \eqref{quadvnf1},
$B=\spann\{1,a,b\}$ is a subalgebra of $A$ of dimension at most $3$. If $\dim B<3$ then
$B$ is commutative and thus von-Neumann finite, and if $\dim B=3$ then $B$ is von-Neumann
finite by assumption. Hence, $ba=1$, and so $A$ is von-Neumann finite. 
Reversibility is proved analogously. 

\eqref{quadvnf3}
In view of \eqref{quadvnf2}, we need to show that a $3$-dimensional quadratic algebra is
von-Neumann finite and reversible if and only if it is commutative. Clearly, any
commutative algebra is both von-Neumann finite and reversible. Let $B=\spann\{1,a,b\}$ be
a von-Neumann-finite reversible quadratic algebra, and $ab=\a+\b a+\gamma b$, for some
$\a,\b,\gamma\in F$. Then
$$(a-\gamma)(b-\b) = ab - \b a -\gamma b +\b\gamma= \a +\b\gamma \in F,$$
implying that $(a-\gamma)(b-\b)=(b-\b)(a-\gamma)$ since $B$ is von-Neumann finite and
reversible.
Now $0=(a-\gamma)(b-\b)-(b-\b)(a-\gamma)= ab-ba$, so $ab=ba$ 
and $B=\spann\{1,a,b\}$ is commutative. 
\end{proof}

Note that in an involutive algebra, the identity $\n(a)=a\bar{a}=\bar{a}a$ holds:
as $\tr(a)\in F$ it follows that $\tr(a)a=a\tr(a)$, and thus
\begin{equation} \label{bareq}
  a\bar{a} = a(\tr(a)-a)=\tr(a)a-a^2 =(\tr(a)-a)a= \bar{a}a \,.
\end{equation}
The equation \eqref{bareq} also implies that $a^2-\tr(a)a+\n(a)=0$ so, in particular,
every involutive algebra is quadratic. 

In the paper \cite{osborn62}, Osborn developed the fundamentals of a theory for
quadratic algebras over fields of characteristic different from $2$. For such an algebra
$A$, set $$\Im A = \{u\in A\setminus F1 \mid u^2\in F1\}\cup\{0\}\,.$$
Then $A$ decomposes as a vector space as $A=F\oplus\Im A = F1\oplus\Im A$.
This decomposition defines an $F$-bilinear form $(\cdot,\cdot)$ and an anti-commutative
multiplication $\ep $ on $\Im A$, by the formula 
$uv=(u,v) + u\ep  v\in F\oplus \Im A$ for all $u,v\in\Im A$. 
Multiplication in $A=F\oplus \Im A$ can now be written as
\begin{equation} \label{quadmult}
 (\a,u)(\b,v) = (\a\b + (u,v),\,\a v+\b u +u\ep  v)\,,
\end{equation}
for $\a,\b\in F$, $u,v\in\Im A$.
A linear map $\s: A\to A,\: a\mapsto\bar{a}$ is defined by $\s(\a,u) =(\a,-u)$ and,
similarly to the involutive case, we have $\tr(a) = a+\bar{a}=2\a\in F$,
$\n(a)=a\bar{a}=\bar{a}a=\a^2-(u,u)\in F$ and $a^2-\tr(a)a+\n(a)=0$ for $a=(\a,u)\in A$.
Moreover, $(u,v)=\frac{1}{2}\tr(uv)$ 
for all $u,v\in\Im(A)$.

\begin{lma}[Osborn {\cite[p~203]{osborn62}}] Let $A$ be a quadratic algebra, and  $\chr F\ne 2$.
  \begin{enumerate} \label{condition}
  \item The algebra $A$ is involutive 
    if and only if $(\cdot,\cdot)$ is symmetric.
        \label{invcondition}
  \item 
    The algebra $A$ is flexible if and only if the bilinear form $(\cdot,\cdot)$ is
    symmetric and $(u,u\ep  v)=0$ holds for all $u,v\in\Im A$.
    \label{flexcondition}
  \end{enumerate}
In particular, every flexible quadratic $F$-algebra is involutive. 
\end{lma}

For $A$ quadratic and $\chr F\ne2$, 
we denote by $\lt_u^{\ep}$ the linear map $\Im A\to \Im A,\:v\mapsto u\ep v$.

\begin{lma} \label{3dimcommsubalg}
Let $A$ be a quadratic $F$-algebra, and $\chr F \ne2$.
\begin{enumerate}
\item 
  Every  $3$-dimensional subalgebra of $A$ is commutative if and only if the condition
  $u\ep  v\in\spann\{u,v\}$ implies $(u,v)=(v,u)$ and  $u\ep v =0$ for all
 $u,v\in \Im A$.%
  \label{3dcsaa}
\item 
  Assume additionally that $F$ is algebraically closed, and $A$ involutive. Then every $3$-dimensional subalgebra of $A$ is commutative if and only if the map $\lt_u^{\ep}$ is nilpotent for all $u,v\in\Im A$.
  \label{3dcsab}
\end{enumerate}
\end{lma}

\begin{proof}
\eqref{3dcsaa}
For any $u,v\in\Im A$, the subspace $B=\spann\{1,u,v\}\subset A$ is a subalgebra if and
only if $u\ep v\in\spann\{u,v\}$. On the other hand, 
$uv-vu= ((u,v)- (v,u) , 2u\ep v)\in F\oplus\Im A$. 
Consequently, $B$ is commutative if and only if $(u,v)=(v,u)$ and $u\ep v =0$.

\eqref{3dcsab}
Since $A$ is involutive, the bilinear form $(\cdot,\cdot)$ is symmetric by
Lemma~\ref{condition}\eqref{invcondition}.
Moreover, as the product $\ep$ on $\Im A$ is anti-commutative, the condition
$u,v\in\spann\{u,v\}\:\Rightarrow\: u\ep v=0$ for all $u,v\in\Im A$ is equivalent to that,
for all $u\in\Im A$, the map $\lt_u^{\ep}:\Im A\to\Im A$ has no non-zero eigenvalues.
By the Jordan normal form theorem, this is equivalent to $\lt_u^{\ep}$ being nilpotent for
all $u\in \Im A$. \
Hence, the result follows from Lemma~\ref{3dimcommsubalg}\eqref{3dcsaa}.
\end{proof}

For the case $\chr F\ne2$, we shall need the following fact about Cayley--Dickson algebras, which is a consequence of the fact that their norms are \emph{Pfister forms} \cite[Chapter~X]{lam05}.
\begin{lma}[{\cite[Theorem~X.1.7]{lam05}}] \label{hyperbolic}
  Assume that $\chr F\ne2$ and let $A$ be a Cayley--Dickson algebra of even dimension $2m$, with isotropic norm. 
  Then $A$ has a totally isotropic subspace of dimension $m$. 
\end{lma}

\section{Proofs of our results}
\label{proofs}

\subsection{Proof of Theorem~\ref{alt}}

\begin{prop} \label{alt1}
Every alternative algebra that is either finite dimensional or
reversible is von-Neumann finite. 
\end{prop}

\begin{proof}
Let $A$ be an alternative algebra, and $B\subset A$ a subalgebra generated by at
most two elements. If $A$ is finite dimensional or reversible then so is $B$. If $B$ is
finite dimensional then it is von-Neumann finite by Lemma~\ref{fdass}; if it is reversible
then the same conclusion follows from the argument given in the introduction.
Thus, by Lemma~\ref{vnfsubalg}, the algebra $A$ is von-Neumann finite.
\end{proof}

\begin{prop} \label{alt3}
Every Hurwitz algebra with anisotropic norm is reversible. A split Hurwitz algebra is
reversible if and only if its dimension is less than or equal to $2$. 
\end{prop}

\begin{proof}
Hurwitz algebras with anisotropic norm have no zero divisors, and Hurwitz algebras of
dimension at most $2$ are commutative. Hence, all algebras of either type are reversible.
If $A$ is a split Hurwitz algebra of dimension at least $4$, then it contains a split
quaternion subalgebra, that is, an algebra isomorphic to $H$. Since $H$ is not reversible,
neither is $A$ in this case.
\end{proof}

\subsection{Proof of Theorem~\ref{quadflex}(\ref{quadflex1},\ref{quadflex2})}

\begin{prop} \label{0div}
  Every algebra without zero divisors, that is either flexible or quadratic, is
  von-Neumann finite.
\end{prop}

\begin{proof}
  Let $A$ be an algebra without zero divisors. 
  If $A$ is flexible and $ab=1$ then $a=(ab)a=a(ba)$, whence $a(1-ba)=0$ and $ba=1$.
  If  $A$ is quadratic then $\spann\{1,a\}$ is a field for every $a\in A$, and hence there
  exists an element $\tilde{a}\in\spann\{1,a\}$ such that $a\tilde{a}=\tilde{a}a=1$. Now,
  if $ab=1$ then $a(b-\tilde{a})= ab -a\tilde{a} = 1-1=0$, implying that $b=\tilde{a}$ and
  hence $ba=1$.
\end{proof}

For the remainder of Section~\ref{proofs}, we assume that $\chr F\ne 2$. 

\begin{prop} \label{anisotropic}
Let $A$ be a flexible quadratic $F$-algebra. If the norm $\n$ of
$A$ is non-degenerate on every $3$-dimensional subalgebra of $A$, then $A$ is von-Neumann
finite and reversible.
\end{prop}

\begin{proof}
Any $3$-dimensional subalgebra $B\subset A$ can be written as $B=\spann\{1,u,v\}$, where $u,v\in\Im A$ are anisotropic vectors. Since $A$ is flexible, the bilinear form $(\cdot,\cdot)$ on $\Im A$ is symmetric by Lemma~\ref{condition}\eqref{flexcondition},
moreover, $\n(w)=-(w,w)$ and $\<w,z\>_{\n}=-2(w,z)$ for all $w,z\in \Im A$.
Again by Lemma~\ref{condition}\eqref{flexcondition}, we have $(u,u\ep v) = (v,u\ep v) =0$ implying that $\<u,u\ep v\> = \<v,u\ep v\>=0$ and,
consequently, $u\ep v\in \Im B\cap (\Im B)^\perp =0$. Thus
$B$ is commutative, by Lemma~\ref{3dimcommsubalg}\eqref{3dcsaa}.
Lemma~\ref{quadvnf}\eqref{quadvnf3} gives the result.
\end{proof}

\subsection{Proof of Theorem~\ref{doubled}}

Since Cayley--Dickson algebras are flexible and quadratic, Proposition~\ref{anisotropic}
implies the following result.

\begin{cor} \label{doubledcor}
All Cayley--Dickson algebras with anisotropic norm are von-Neumann finite and reversible.
\end{cor}

Theorem~\ref{doubled}\eqref{doubled2} is a consequence of the following proposition.

\begin{prop} \label{splitcd}
  Let $A$ be a Cayley--Dickson algebra of dimension at least $4$ with isotropic norm. Then
  $A$ has a subalgebra isomorphic to the split quaternion algebra $H$.
\end{prop}

\begin{proof}
By construction, the algebra $A$ contains a subalgebra of the form
$B'=\V_{\mu}(B)=B\oplus B\ell$, where $B$ is a Cayley--Dickson algebra with anisotropic
norm, $\mu\in F\setminus\{0\}$, and the norm of $B'$ is isotropic. 
Set $m=\dim B$. If $m=1$ then $B'\simeq F\times F$ and, since $\dim A\ge4$, there exists a
subalgebra of $A$ of the form $\V_{\nu}(B')\simeq \V_{\nu}(F\times F)\simeq H$,
$\nu\in F\setminus0$.

Assume instead that $m\ge2$. 
By Lemma~\ref{hyperbolic}, the $(2m)$-dimensional quadratic space $B'$ contains a
totally isotropic subspace $U$ of dimension $m$. 
It follows that the $(m+1)$-dimensional subspace $B\oplus F\ell$ of $B'$ intersects $U$
non-trivially, and hence $B\oplus F\ell$ contains an isotropic vector $v$. 
Let $b\in B$ be such that $v-b\in F\ell$, and let $C\subset B$ be a $2$-dimensional
subalgebra containing $b$. 
Then $C'=C\oplus C\ell\subset B'$ is a subalgebra of dimension
$4$, and $C'\simeq\V_{\mu}(C)$. Hence $C'$ is a quaternion subalgebra of $A$, and since
$v\in C'$ is an isotropic element, it follows that $C'\simeq H$. 
\end{proof}

\begin{cor} \label{nonrevcd}
A Cayley--Dickson algebra with isotropic norm is reversible 
if and only if its dimension is at most two.
\end{cor}

\subsection{Proof of Theorem~\ref{quadflex}\eqref{quadflex3} and Theorem~\ref{inv}}

We start by proving Theorem~\ref{inv}.

\begin{lma} \label{vnfcomm}
  Let $A$ be an involutive algebra.
  \begin{enumerate}
  \item The identity $\overline{ab} -ba =\tr(a)\tr(b) - \tr(a)b - \tr(b)a$ holds for all
    $a,b\in A$. 
    \label{vnfcomm1}
  \item The algebra $A$ is von-Neumann finite if and only if the condition 
    $ab\in F\setminus\{0\}$ implies that either $a,b\in\Im A$ or $1,a,b$ are linearly
    dependent. 
    \label{vnfcomm2}
  \item The algebra $A$ is reversible if and only if $ab=0$ implies that either 
    $a,b\in\Im A$ or $1,a,b$ are linearly dependent. 
\label{vnfcomm3}
  \end{enumerate}
\end{lma}

\begin{proof}
\eqref{vnfcomm1}
This is a straightforward calculation:
\begin{align*}
  \overline{ab} - ba = \bar{b}\bar{a} - ba = (\tr(b)-b)(\tr(a)-a) - ba = \tr(b)\tr(a) - \tr(b)a -
  \tr(a)b \,.
\end{align*}

(\ref{vnfcomm2},\ref{vnfcomm3})
We shall prove that if $ab\in F$, then $ab=ba$ if and only if either $1,a,b$ are
linearly dependent or $a,b\in\Im A$. This, together with Lemma~\ref{vnfsubalg},
implies the result.

Let $ab\in F$. Then $\overline{ab} = ab$, so 
$ab-ba = \overline{ab} - ba = \tr(a)\tr(b) - \tr(a)b - \tr(b)a$ by \eqref{vnfcomm1}.
Hence, if $ab=ba$ then either $1,a,b$ are linearly dependent, or $\tr(a)=\tr(b)=0$, that
is $a,b\in\Im A$. Conversely, if $\tr(a)=\tr(b)=0$ then clearly $ab-ba=0$. If instead
$1,a,b$ are linearly dependent, then these elements are contained in a subalgebra of
dimension at most two, and hence commute with each other.
\end{proof}

\begin{lma} \label{basis}
  Let $B$ be a $3$-dimensional non-commutative involutive algebra. Then there exists a
  basis $u,v$ of $\Im B$ such that $u\ep  v = u$ and either $(u,u)=0$ or $(u,v)=0$. 
\end{lma}

\begin{proof}
Since $B$ is involutive, the bilinear form $(\cdot,\cdot)$ is symmetric. Consequently,
commutativity of $B$ is equivalent to the condition that $x\ep  y=0$ for all 
$x,y\in\Im B$.  
As $B$ is not commutative, the product $\ep $ on $\Im B$ is non-trivial, so there exist
$x,y\in\Im B$ such that $x\ep  y\ne0$. Set $u=x\ep  y$. Using the anti-commutaticity of
$\ep$ we compute
$$(\a x+\b y)\ep (\gamma x + \d y) = \a\gamma x\ep  x + \a\d x\ep  y +
\b\gamma y\ep  x +\b\d y\ep  y =(\a\d-\b\gamma)x\ep  y =
(\a\d-\b\gamma)u \,. $$ 
This means that $(\Im B)\ep (\Im B) = \spann\{u\}$, and that $w\ep z=0$ if and
only if $w,z\in\Im A$ are linearly dependent. Now, if $(u,u)\ne0$ then there exists a
vector $v\in\Im B$ such that $(u,v)=0$ and $u\ep  v=u$, and thus $u,v$ is a basis of $\Im B$
with the required properties. 
In case $(u,u)=0$, any $v\in\Im B$ such that $u\ep  v=u$ will do. 
\end{proof}

\begin{prop} \label{invvnf}
  Let $A$ be an involutive algebra. The following statements are equivalent.
  \begin{enumerate} \renewcommand{\theenumi}{\roman{enumi}}
  \item The algebra $A$ is von-Neumann finite;
    \label{invvnf1}
  \item every $3$-dimensional subalgebra of $A$ is either commutative or associative;
    \label{invvnf2}
  \item every $3$-dimensional subalgebra of $A$ is either commutative or isomorphic to the
    algebra $U$ of upper triangular $2\times2$ matrices. 
    \label{invvnf3}
  \end{enumerate}
\end{prop}

\begin{proof}
As all commutative algebras and all finite-dimensional associative algebras are
von-Neumann finite, the implication \eqref{invvnf2}$\Rightarrow$\eqref{invvnf1} is
immediate from Lemma~\ref{quadvnf}\eqref{quadvnf2}. Moreover, we have
\eqref{invvnf3}$\Rightarrow$\eqref{invvnf2} since the algebra $U$ is associative.
To prove \eqref{invvnf1}$\Rightarrow$\eqref{invvnf3}, by
Lemma~\ref{quadvnf}\eqref{quadvnf2}, it suffices to show that if $B$ is a $3$-dimensional
von-Neumann-finite non-commutative involutive algebra, then $B\simeq U$.

By Lemma~\ref{basis}, there exists a basis $u,v$ of $\Im B$ such that $u\ep  v=u$.
For this basis, we have 
$$(0,u)(-1,v) = ((u,v),-u+u\ep  v) = (u,v)\in F.$$ 
Moreover, the elements $1,(0,u), (-1,v)$ of $B$ are linearly independent, 
$(-1,v)\notin\Im B$ and $B$ is von-Neumann finite.
Thus Lemma~\ref{vnfcomm}\eqref{vnfcomm2} implies that $(u,v)=(0,u)(-1,v) = 0$.
Similarly, again applying Lemma~\ref{vnfcomm}\eqref{vnfcomm2},
\begin{align*}
  (1,u+v)(-1,v) &= (-1+(v,v), v - u-v+u\ep  v) = (v,v)-1\in F 
  &\Rightarrow&\quad 
  (v,v)=1\,,\;\mbox{and} \\
  (1,u+v)(0,u) &= ((u,u), u + v\ep  u) = (u,u)\in F
  &\Rightarrow&\quad
  (u,u)=0\,.  
\end{align*}
From the identities $u\ep  v=u$, $(u,u)=(u,v)=0$, $(v,v)=1$, and the equation
\eqref{quadmult}, it is easy to work out the multiplication table of $B$: 
$u^2 =0$, $v^2 = 1$ and $uv=u=-vu$. 
This coincides with the multiplication table of $U$ given in Table~\ref{Umult} 
(see Example~\ref{HU} below), consequently, $B\simeq U$. 
\end{proof}

\begin{prop} \label{invrev}
Let $A$ be an involutive algebra. The following statements are equivalent.
\begin{enumerate}\renewcommand{\theenumi}{\roman{enumi}}
\item The algebra $A$ is reversible. \label{invrev1}
\item Every $3$-dimensional subalgebra of $A$ is commutative. \label{invrev2}
\end{enumerate}
\end{prop}

\begin{proof}
Again, the implication \eqref{invrev2}$\Rightarrow$\eqref{invrev1} follows directly from
Lemma~\ref{quadvnf}\eqref{quadvnf2}.
For the converse, assume that there exists a reversible involutive algebra $B$ of
dimension $3$ that is not commutative. 
Then, by Lemma~\ref{basis}, there exist $u,v\in\Im B$ such that $u\ep  v=u$ and either
$(u,u)=0$ or $(u,v)=0$. 
Computing
$$(1,v)(0,u) = ((u,v), u+v\ep  u)=(u,v)\in F$$
and invoking Lemma~\ref{vnfcomm}\eqref{vnfcomm3}, we get that $(u,v)\ne0$ and, consequently, $(u,u)=0$. 
Now, similarly,
\begin{align*}
  (1,u+v)(1,u-v) &= (1-(v,v),\,(u-v)+(u+v) + (u+v)\ep (u-v)) \\
  &= (1-(v,v),\,2u-u\ep  v +v\ep  u) = 1-(v,v)\in F
\end{align*}
and thus $1-(v,v)\ne0$ by Lemma~\ref{vnfcomm}\eqref{vnfcomm3}.
Now, setting $y= (1-(v,v))u + (u,v)v$ we get
\begin{align*}
(v,y) &= (1-(v,v))(u,v) + (u,v)(v,v) = (u,v),\\
v\ep  y &= (1-(v,v))v\ep  u = ((v,v)-1)u, \\
\intertext{and hence} 
  (1,v)(-(u,v), y) &= 
(-(u,v) + (v,y) ,\, y -(u,v)v + v\ep  y) \\
&=  ( 0,\, (1-(v,v))u + (u,v)v -(u,v)v + ((v,v)-1)u ) = 0 \,.
\end{align*}
Note that since $(1-(v,v))\ne0$, the elements $v$ and $y$ are linearly independent. As
$(1,v)\notin\Im B$, Lemma~\ref{vnfcomm}\eqref{vnfcomm3} implies that $B$ cannot be
reversible, a contradiction. Hence $B$ is commutative.
\end{proof}

As the norm of $U$ is isotropic, Proposition~\ref{invvnf} and Proposition~\ref{invrev}
together have the following consequences. 

\begin{cor} \label{revisvnf}
Let $A$ be an involutive algebra.
  \begin{enumerate}
  \item If $A$ is reversible then it is von-Neumann finite.
    \label{revisvnf1}
  \item If $A$ is von-Neumann finite and the norm of $A$ is anisotropic, then $A$ is
    reversible.
    \label{revisvnf2}
  \end{enumerate}
\end{cor}

The following result establishes the proof of Theorem~\ref{quadflex}\eqref{quadflex3}.

\begin{prop}\label{algclosedrev}
  Assume that $F$ is algebraically closed, and let $A$ be a flexible and quadratic
  $F$-algebra. Then $A$ is reversible if and only if one of the following conditions
  holds:
  \begin{enumerate}\renewcommand{\theenumi}{\roman{enumi}}
  \item $u\ep v=0$ for all $u,v\in \Im A$, that is, $A$ is commutative;%
    \label{ep=0}
  \item $\n(u)=0$ for all $u\in \Im A$, and the map 
    $\lt_u^{\ep}:\Im A\to\Im A,\: v\mapsto u\ep v$ is nilpotent for all $u\in\Im A$.%
    \label{n=0}
  \end{enumerate}
\end{prop}
Note that in the latter case, the subspace $\Im A$ is an ideal in $A$, and $A$ is
  obtained from $\Im A$ by adjoining an identity element. Hence, reversible
algebras of this type are given by anti-commutative (non-unital) algebras $V$ satisfying
the nilpotency condition specificed in \eqref{n=0}. This holds, for example, if $V$ is a
nilpotent Lie algebra.

\begin{proof}
Combining Proposition~\ref{invrev} with Lemma~\ref{3dimcommsubalg}\eqref{3dcsab}, we see
that $A$ is reversible if and only if $\lt_u^{\ep}$ is nilpotent for all $u\in \Im A$.
Clearly, this holds if $A$ satisfies either of the conditions \eqref{ep=0} and \eqref{n=0}. 
For the converse implication, assume that $\lt_u^{\ep}$ is nilpotent for all $u\in\Im A$.
We shall show that for any anisotropic $w\in\Im A$, the identity $\lt_w^{\ep}=0$ holds.
Then, since the anisotropic vectors form a Zariski-open subset of $\Im A$, whereas the condition $\lt_w^{\ep}=0$ is closed, it follows that either $\n=0$ or $\lt_w=0$ for all $w\in\Im A$. 

First, linearising the identity $(u,u\ep v)=0$ in
Lemma~\ref{condition}\eqref{flexcondition} gives the equivalent condition
\begin{equation}
(u\ep v,w)=(u,v\ep w)
\end{equation}
for all $u,v,w\in\Im A$.
Assume that $w\in\Im A$ is anisotropic, $0\le m <\dim(\Im A)$, and that there exist
linearly independent
anisotropic vectors $w=v_1,v_2,\ldots,v_m$, such that $w\ep v_i=0$
for all $i\le m$. 
Set $W_m=\spann\{v_1,\ldots,v_m\}^{\perp}\subset\Im A$. 
Then, for all $u\in\Im A$,
$$0=(w\ep v_i,u) = -(v_i\ep w, u) = -(v_i,w\ep u),$$
that is, $\lt_w^{\ep}(\Im A)\subset W_m$.
In particular, the subspace $W_m$ is invariant under $\lt_w^{\ep}$. 
The map $\lt_w^{\ep}$ being nilpotent, it means that there exists a non-zero element 
$z\in W_m$ such that $\lt_w^{\ep}(z)=0$.

Now, set $v_{m+1}=z$ if $\n(z)\ne0$, and $v_{m+1}=z+w$ if $\n(z)=0$. In either case, the
elements $v_1,\ldots,v_{m+1}\in\Im A$ are linearly independent, anisotropic, and 
contained in $\ker(\lt_w^{\ep})$. By induction, $\ker(\lt_w^{\ep})=\Im A$. 
\end{proof}

\section{Examples}
\label{examples}

We start by mentioning two particular, associative algebras that are of certain
importance for our study.
\begin{ex} \label{HU}
  Let $\chr F\ne2$. 
  The split quaternion $F$-algebra $H$ is, as already mentioned, isomorphic to the
  algebra $F^{2\times2}$ of $2\times2$ matrices with entries in $F$.
  Choosing
  \[
  i=\pmatr{-1&0\\0&1},\quad j=\pmatr{0&1\\1&0},\quad\mbox{and}\quad k=\pmatr{0&1\\-1&0}
  \]
  gives a basis $\underline{e}=(1,i,j,k)$ of $H$, with multiplication as in
  Table~\ref{Hmult}.

It is easy to see that $\Im H=\spann\{i,j,k\}$, 
and that $\underline{e}$ is an orthogonal basis of $H$.
Indeed, since
$\<x,y\>= \n(x+y)-\n(x)-\n(y) = xy+yx$ for all $x,y\in\Im H$, a subset of $\Im H$ is
orthogonal if and only its elements pair-wise anti-commute -- as do $i,j,k$.

Now consider the subalgebra $U$ of $F^{2\times2}$, consisting of all upper triangular
matrices. Being a subalgebra of an involutive associative algebra, $U$ is again 
involutive and associative.
We choose a basis $\underline{f}=(1,u,v)$, where
$$u=\frac{1}{2}(j+k)= \pmatr{0&1\\0&0}\qquad\mbox{and}\qquad v=i=\pmatr{-1&0\\0&1}\,.$$
Again, $\underline{f}$ is orthogonal, and $\Im U = \spann\{u,v\}$.
The multiplication of elements in $\underline{f}$ is given by Table~\ref{Umult}.

  \begin{table}[htb] 
\begin{minipage}[b]{0.45\linewidth}
    \[
    \begin{array}{c|cccc} 
      \cdot& 1&i&j&k\\
      \hline
      1& 1&i&j&k\\
      i& i&1&k&j\\
      j& j&-k&1&-i\\
      k& k&-j&i&-1
    \end{array}
\]
   \caption{Multiplication in $H$.} \label{Hmult}
\end{minipage}
\quad
\begin{minipage}[b]{0.45\linewidth}
    \[
    \begin{array}{c|ccc} 
      \cdot& 1&u&v\\
      \hline
      1&1&u&v\\
      u& u&0&u \\
      v& v&-u&1
    \end{array}
\]\vspace*{.8ex}
\caption{Multiplication in $U$.} \label{Umult}
\end{minipage}
  \end{table}

As $(1+v)u=0$ but $u(1+v)=2u\ne0$, the algebras $U$ and $H$ are not reversible. However,
being finite-dimensional associative algebras, they are von-Neumann finite.
\end{ex}

The following examples demonstrate the necessity of some of the assumptions in our main theorems.
First we give an example of a flexible quadratic algebra with non-degenerate norm that is neither reversible nor von-Neumann finite. The existence of such an algebra implies that the hypothesis in Theorem~\ref{quadflex}\eqref{quadflex2}, that the norm is non-degenerate on every $3$-dimensional subalgebra, cannot be replaced by the weaker condition of the norm being non-degenerate on $A$ itself.

\begin{ex} \label{flexquadnonvnf}
Assume that $\chr F\ne2$, and let $H$ be the split quaternion algebra with basis
$\underline{e}=(1,i,j,k)$ as in Example~\ref{HU}.

Set $A=H\oplus Fl$, where $l(\Im H)=(\Im H) l =0$ and $l^2=-1$. This defines $A$ as a
quadratic algebra with $\Im A = \spann\{i,j,k,l\}$.
Since $H$ is associative and thus in particular flexible, the bilinear form
$(\cdot,\cdot)_H$ on $H$ is symmetric, and $(u,u\ep v)_H=0$ for all
$u,v\in\Im H$. By the construction, it follows that also $(\cdot,\cdot)_A$ is symmetric,
and $(w,w\ep z)_A=0$ for all $w,z\in\Im A$. Hence, Lemma~\ref{condition} implies
that $A$ is flexible and thus, in particular, involutive. The elements $1,i,j,k,l$
constitute an orthogonal basis of $A$, so the norm on $A$ is non-degenerate.

Set $x=i+l$, $y=j+k$, and $B=\spann\{1,x,y\}\subset A$. 
We have 
\begin{equation*}
  xy=(i+l)(j+k) = ij+ik=k+j=y \quad\mbox{and}\quad 
  yx=(j+k)(i+l)=ji+ki= -k-j=-y
\end{equation*}
so $B$ is a
$3$-dimensional non-commutative subalgebra of $A$.
Moreover, $x^2=i^2+l^2=1-1=0$, so $x^2y=0$, whereas $x(xy)=xy=y\ne x^2y$. Hence $B$ is
not associative either. From Theorem~\ref{inv} it now follows that the algebra $A$ is not
von-Neumann finite, and not reversible.
Indeed, straightforward calculations shows that
\begin{align*}
(1-x-y)(1+x-y)&=1, \qquad 
& y(1+x)&=0,
\\
(1+x-y)(1-x-y) &= 1+4y\,,
& (1+x)y&=2y\,,
\end{align*}
giving explicit counterexamples.

We remark that $B^\perp\cap B=\Im B$, so indeed, the algebra $A$ does not satisfy the
condition in Theorem~\ref{quadflex}\eqref{quadflex2} that the norm be non-degenerate on
every $3$-dimensional subalgebra. 
\end{ex}

Dropping the condition of flexibility, there exist involutive algebras even with
anisotropic norm that are not reversible and not von-Neumann finite.

\begin{ex}
Assume that $\chr F\ne2$. Let $V$ be a $2$-dimensional vector space over $F$, and 
$q:V\to F$ a quadratic form such that the form 
$1\perp(-q):F\oplus V\to F,\:(\a,x)\mapsto\a^2-q(x)$ is anisotropic (\eg, in the real case,
$q$ is negative definite).
Take an orthogonal basis $(u,v)$ of $V$, and set $u\ep v=-v\ep u=u$ and 
$(x,y) = \frac{1}{2}\<x,y\>_q = \frac{1}{2}\left(q(x+y)-q(x)-q(y) \right)$ 
for all $x,y\in V$. 
This defines a quadratic algebra structure on $A=F\oplus V$, with $\Im A=V$ and
multiplication given by the equation \eqref{quadmult}. 

Clearly, $(\cdot,\cdot)$ is symmetric, so $A$ is involutive by
Lemma~\ref{condition}\eqref{invcondition}. Moreover, 
$$\n(\a,x) = (\a,x)\overline{(\a,x)} = (\a,x)(\a,-x) = \a^2-(x,x) = \a^2-q(x),$$
meaning that $\n$ is anisotropic.
Now,
\begin{align*}
  (0,u)(-1,v) &= ((u,v),\, -u+u\ep v) = 0, \\
  (-1,v)(0,u) &= ((v,u),\, -u+v\ep u) = (0,-2u)\ne 0,\\
\intertext{so $A$ is not reversible; and}
(0,u)(-1,u+v) &= (0,u)(-1,v) + (0,u)^2 = -\n(u)\ne0,\\
(-1,u+v)(0,u) &= (-1,v)(0,u) + (0,u)^2 = (-\n(u), -2u) \ne (0,u)(-1,u+v),
\end{align*}
implying that $A$ is not von-Neumann finite either. 
It follows from Theorem~\ref{quadflex}\eqref{quadflex2} that $A$ cannot be flexible,
which indeed also can be seen directly from the identities $(uv)u = u^2=-u(-u)=-u(vu)$.
\end{ex}

For associative algebras, reversibility implies von-Neumann finiteness, and by 
Theorem~\ref{inv}, the same is true for involutive algebras over fields of
characteristic different from $2$.
However, this implication does not hold for general non-associative algebras, as the
following example demonstrates.

\begin{ex} \label{revnotvnf}
Let $A$ be any non-commutative non-quadratic division algebra of dimension $3$ over $F$.
(Non-quadraticity is automatic if $\chr F\ne2$, by \cite[Theorem~3]{osborn62}.)
For example, $A$ may be a \emph{twisted field} of dimension $3$ over a finite field $F$
with at least $3$ elements (see \cite{albert58}).

As $A$ is not quadratic, there exists some element $a\in A$ such that
$A=\spann\{1,a,a^2\}$, and non-commutativity then implies that $aa^2\ne a^2a$.
Since $A$ is a division algebra, there exists an element $b\in A$ such that $ab=1$.
If $b\in\spann\{1,a\}$ then $1=ab\in\spann\{a,a^2\}$,
which is impossible since $1,a,a^2$
are linearly independent. 
Hence $b\notin\spann\{1,a\}$ and therefore $ba\ne ab$. Thus, $A$ is not
von-Neumann finite. But $A$ is reversible, since it is a division algebra. 
\end{ex}

\end{document}